\documentclass[12pt]{amsart}
\usepackage{graphicx}
\usepackage{amssymb}
\usepackage{enumerate}
\usepackage[bookmarks=true]{hyperref}
\hypersetup{colorlinks,citecolor=blue,linkcolor=blue}

\newtheorem{theorem}{Theorem}[section]
\newtheorem{lemma}[theorem]{Lemma}

\newtheorem{corollary}[theorem]{Corollary}
\newtheorem{definition}[theorem]{Definition}
\newtheorem{rmk}[theorem]{Remark}

\DeclareMathOperator{\SL}{SL}

\DeclareMathOperator{\PSL}{PSL}
\newcommand{\R}{\mathbb{R}}
\newcommand{\hyp}{\mathbb{H}}
\newcommand{\Z}{\mathbb{Z}}
\newcommand{\C}{\mathbb{C}}
\newcommand{\Q}{\mathbb{Q}}

\newcommand{\F}{\mathcal{F}}

\newcommand{\SLOd}{\mathrm{PSL}(2, \mathcal{O}_d)}

\newcommand{\SLC}{\mathrm{SL}(2, \mathbb{C})}
\newcommand{\Od}{\mathcal{O}_d}
\newcommand{\Gd}{\Gamma_d}

\newcommand{\hip}{\mathbb{H}^3}
\newcommand{\hipd}{\mathbb{H}^2}

\newcommand{\FD}{\mathcal{F}}
\newcommand{\Ht}{\mathrm{H}}
\newcommand{\B}{\mathcal{B}}
\newcommand{\pd}{\mathcal{P}_d}

\newcommand{\Nd}{\mathrm{N}_d}
\newcommand{\Xd}{\mathrm{X}_d}
\newcommand{\Pd}{\mathbb{P}^1K_d}
\newcommand{\N}{\mathrm{N}}
\newcommand{\Ud}{\mathrm{U}_d}
\newcommand{\W}{\mathrm{W}}
\newcommand{\f}{\mathrm{F}}

\begin{document}
\title{Height estimates for Bianchi groups}

\author{Cayo D\'oria}\thanks{D\'oria is grateful for the support of FAPESP grant 2018/15750-9.}
\author{Gisele Teixeira Paula}\thanks{Paula was partially supported by Math-AmSud Project 88887.199703/2018-00 from CAPES}
\address{
Departamento de Matem\'atica Aplicada, IME-USP\\ Rua do Mat\~ao, 1010, Cidade Universit\'aria\\
05508-090, S\~ao Paulo SP, Brazil.}
\email{cayofelizardo@ime.usp.br}
\address{
Departamento de Matem\'atica, DMAT/CCE/UFES\\
Av. Fernando Ferrari, 514, Campus de Goiabeiras \\  29075-910, Vit\'oria ES, Brazil. 
}
\email{gisele.paula@ufes.br}

\begin{abstract}
We study the action of Bianchi groups on the hyperbolic $ 3- $space $ \hyp^3 $. Given the standard fundamental domain for this action and any point in $ \hyp^3 $, we show that there exists an element in the group which sends the given point into the fundamental domain such that its height is bounded by a quadratic function  on  the  coordinates of the point. This generalizes and establishes a sharp version of a similar result of Habegger and Pila for the action of the Modular group on the hyperbolic plane. Our main theorem can be applied in the reduction theory of binary Hermitian forms with entries in the ring of integers of quadratic imaginary fields. We also show that the asymptotic behavior of the number of elements in a fixed Bianchi group with height at most $T$ is biquadratic in $T$.

\end{abstract}

\maketitle
\section{Introduction}
The theory of lattices in the Lie group $\PSL(2,\mathbb{C})$ is a source of important problems in many branches of mathematics, e.g.  complex variables, dynamical systems, group theory, hyperbolic geometry and number theory. Among such lattices, the class of arithmetic groups attracts a special attention because of their remarkable properties. A hyperbolic $3$-orbifold is called arithmetic if it is the quotient of the hyperbolic $3$-space by an arithmetic lattice. We recall that the closed hyperbolic manifold of minimal volume, the Weeks manifold, and the cusped hyperbolic $3$-manifold of minimal volume are arithmetic manifolds. 

An important subclass of arithmetic lattices of $\PSL(2,\mathbb{C})$ is that of the \emph{Bianchi groups}. A basic fact about such groups is that any non-uniform arithmetic lattice of $\PSL(2,\mathbb{C})$ is commensurable to some Bianchi group \cite[Thm. 8.2.3]{MR03}. These groups have been studied since the seminal work of Bianchi \cite{bianchi} as the natural generalization of the modular group $\PSL(2,\mathbb{Z})$ and until today there are some deep open questions about them. For example, there exist explicit formulaes for the covolume of Bianchi groups, but it is not known whether the quotient of the covolumes of any two Bianchi groups is a rational number \cite{Borel}. Another important question \cite[Section 7.6]{elstrodt} is a particular case for a conjecture raised by Selberg which asks if the spectral gap in the discrete spectrum of any Bianchi orbifold is at least $1.$ 

For any lattice in $\PSL(2,\mathbb{C})$, there exists a measurable \emph{fundamental domain} in the hyperbolic $3$-space. The usefulness of a fundamental domain is not limited to giving the covolume. There is a lot of geometric and algebraic information which can be extracted from these sets. For example, finite presentation for any lattice can be obtained from its fundamental domain \cite[paragraph 4.7]{M15}. Although the existence of these sets is always guaranteed, there are two issues involving them. Firstly, it is not easy in general to find explicit fundamental domains for arbitrary lattices, and secondly, even if some fundamental domain is given, it may not be the most useful for particular applications. Nevertheless, we can approximate the  fundamental domain by \emph{fundamental sets}, which are measurable sets whose projection on the quotient space is surjective and finite-to-one.

This paper is part of a project aimed to describe fundamental sets for Bianchi groups with good geometric properties which are more convenient for some applications than the corresponding fundamental domains.

An example of application of fundamental sets was given in the proof of the Selberg conjecture for Bianchi orbifolds of low volume \cite{EGM89}. In the proof, two key properties are the facts that the fundamental set which they constructed is a product of a rectangle with a segment in the hyperbolic $3$-space and that it is  covered by few copies of an specific fundamental domain.  More generally, let $\mathcal{D}$ be a fundamental domain for a lattice $\Gamma$ and let $\Sigma \supset \mathcal{D}$ be a fundamental set. The set $\{ \gamma \in \Gamma \mid \gamma \mathcal{D} \cap \Sigma \neq \emptyset \}$ has  finite cardinality $N \geq 1$. We say that $\Sigma$ is a \textit{good} approximation of $\mathcal{D}$ if $N$ is \textit{small}.   

In the general theory of lattices in real points of semisimple Lie Groups, \emph{Siegel sets} are fundamental sets with a simple geometric description. On the other hand, in the paper of the second author \cite{gisele} it was shown that in general the approximation by Siegel sets may not be so good from the point of view of the discussion above. 

Inspired by the good geometric shape of Siegel sets we are interested in, the natural candidates for fundamental sets in $\hip$ are given by finite unions of sets of the form $h \Sigma$, where $h \in \PSL(2,\mathbb{C})$ and $\Sigma= R\times [t, +\infty)$ ,  $ R $ is a polygon in $ \C $ and $ t>0 $, both depending on the group. Since we need to count elements in the Bianchi group, there exists a useful tool in number theory, the  \emph{height} of a matrix with entries in a number field. For any real positive constant $L$, the number of elements in the Bianchi group of height at most $L$ is finite. The naive idea is therefore to estimate from above in terms of our fundamental set the height of any element $\gamma$ in the Bianchi group satisfying $\gamma \mathcal{D} \cap \Sigma \neq \emptyset$, where $\Sigma$ is our candidate for fundamental set and $\mathcal{D}$ is an explicit fundamental domain for the Bianchi group. After this, we need to be able to count the number of matrices with bounded height.

It is worth to mention that this idea of giving an upper bound for the height of matrices with the property of finite intersection for fundamental sets appeared in Orr's paper \cite{martinorr} motivated by a previous result of Habegger and Pila \cite{pila}. Our main result is a natural generalization of their height estimate for the similar problem in $ \hipd $.

This paper is organized as follows. In Section \ref{background} we define the main objects of study along this work. In Section \ref{sec3} we present the main result, where we give an estimate for the height of some group element sending a point in $ \hip $ into the fundamental domain of the Bianchi group in terms of special coordinates of this point. In Section \ref{reductionhforms} we give an application of the main theorem in the theory of reduction of binary Hermitian forms with coefficients in the ring of integers of a quadratic imaginary number field. For the sake of completeness, in Section \ref{counting} we prove an analogue of Duke, Rudinick and Sarnak's result \cite{sarnak} on giving the asymptotic growth of the cardinality of the set of matrices with height smaller than a given constant.

\section{Background}
\label{background}

\subsection{Geometry of the hyperbolic 3-space}	\label{geometryh3}

The hyperbolic 3-space $ \hip $ is the unique 3-dimensional connected and simply connected Riemannian manifold with constant sectional curvature equal to -1. We use the upper half-space model of $ \hip $,  which in its properties closely resembles the well-known upper half-plane  model $ \hipd $ of plane hyperbolic geometry. 
$$\hip = \{(z,t) \mid z = x+iy \in \C, t>0\}.$$

We can think of $ \hip $ as a subset of Hamilton's quaternions, by writing $ (z,t) = z+tj $, where $ z=x+iy $ and $ 1,i,j,k $ is the usual basis of the quaternion space.

The hyperbolic metric in $ \hip $ is $ds^2 = \frac{dx^2 + dy^2 +dt^2 }{t^2}.$ In the sense of Riemannian geometry this metric gives rise to 
the hyperbolic volume measure $ \nu$ with corresponding volume element 
$$d\nu = \frac{dxdydt}{t^3}.$$

Topology in $ \hip $  is induced from $ \R^3 $, but  geometry is hyperbolic. The geodesic lines are half-circles or half-lines in $ \hip $ which are orthogonal to the boundary plane $ \C $ in the Euclidean sense. The geodesic surfaces are Euclidean hemispheres or half-planes which again are orthogonal to the boundary $ \C $.

The group $ G = \mathrm{SL}(2,\C) $ acts by isometries on the upper half space $ \hip $. This action can be described as follows: if we represent a point $ P \in \hip $ as a  quaternion $ P= (z,t) = z+tj $, then the action of $ M  =  \left( \begin{array}{cc}
a & b  \\
c & d  \\
\end{array} \right) \in G $ on $ \hip $ is given by
\begin{equation}
\label{actionquaternions}
P\mapsto MP:= (aP+b)(cP+d)^{-1},
\end{equation}
where the inverse is taken in the skew field of quaternions. A direct  
computation in quaternionic coordinates shows that this defines an action of $ \mathrm{SL}(2, \C) $ on $ \hip $. More explicitly, Equation \eqref{actionquaternions} may be written in the form 
\begin{equation}
\label{actionh3}
M(z, t) = 
\displaystyle \Bigg( \frac{(az+b)(\bar{c}\bar{z}+\bar{d}) + a\bar{c}t^2}{\left|cz+d\right|^2 +\left|c\right|^2 t^2}, \frac{t}{\left|cz+d\right|^2 +\left|c\right|^2 t^2} \displaystyle \Bigg).
\end{equation}

The action of $ G $ on $ \hip $ extends to an action on $\hip \cup \partial \hip,$ where $ \partial \hip = \mathbb{P}^1 \C = \C \cup \{\infty\}$. On $\partial \hip$ the action may be described by a simple formula. We represent an element of $ \mathbb{P}^1 \C $ by $ (x: y) $, where $ (x, y) \in  \C^2-\{(0,0)\} $, and represent $ \infty \in \mathbb{P}^1 \C $  by  $(1:0) $. Then the action of $ M $ on $ \mathbb{P}^1 \C $ is given by
\begin{equation}
\label{boundary}
(x: y) \mapsto M(x:y) := (ax + by: cx + dy). 
\end{equation}

For any $M \in G$, the matrix $-M \in G$ and $M(P)=-M(P)$ for all $P \in \hip \cup \partial \hip$. Therefore we can consider the action of the group $\mathrm{PSL}(2,\C)=\SLC / \{I,-I\}$, where $I$ is the identity matrix. We will denote the class of a matrix $M \in \SLC$  in $\mathrm{PSL}(2,\mathbb{C})$ by $[M]$.

\subsection{Bianchi groups}
\label{bianchidefinition}

Given a squarefree integer $d>0$, consider the imaginary quadratic number field $K_d= \Q(\sqrt{-d})$ and let $\Od$ be the ring of integers of $K_d$, $$ \Od  = \Z[\omega]= \left\{a+b \omega \mid a,b \in \Z\right\}, $$ where 
 $\omega =  \left\{ \begin{array}{cc}
 \dfrac{-1 + \sqrt{-d}}{2}, & \mbox{ if } d \equiv 3 \mbox{ (mod } 4);  \\
 & \\
\sqrt{-d}, & \mbox{otherwise} .
\end{array} \right. $

\begin{definition}
	For each $ d $, we define the Bianchi group: $$\Gamma_d := \SLOd = \{M \in \SLC; M_{ij} \in \Od \} / \{ I, -I\}.$$ 
\end{definition}

\begin{definition}
	The  \textit{cusps} $C_{\Gd}$ of $\Gd$ are the elements $[x,y] \in \mathbb{P}^1\C$, for which the stabilizer in $\Gd$ of the action in Equation \eqref{boundary} contains a free abelian group of rank 2. 
\end{definition}

The set of cusps of the Bianchi group $ \Gamma_d $ is given by $C_{\Gd} = \mathbb{P}^1K_d \subset \mathbb{P}^1 \C$ as it can be seen in \cite[Proposition 2.2, Chapter 7]{elstrodt}.

The quotient of the induced action of $\Gd$ on $C_{\Gd}$ is a finite set, which has exactly $ h(d) $ elements, where $ h(d) $ is  the \textit{class number} of the field $ K_d $, i.e., the cardinality of the ideal class group $ \mathcal{I}_d  $ of $ K_d $ (see \cite[Thm 2.4., Chapter 7]{elstrodt}). 

\subsection{The fundamental domains for Bianchi groups}\label{fundamentaldomain}
In \cite[Section 7.3]{elstrodt} a complete geometric description of the construction of fundamental domains for the action of Bianchi groups on $\hip$ is given.

We define the sets
$$\B_d=\{(z,t) \in \hip \mid |cz+d|^2+|c|^2t^2 \geq 1  \mbox{ for all }  c, d \in \Od \, \, , \langle c, d \rangle = \Od \}$$ 
$$\mathcal{P}_1 = \left\{ x+iy  \mid |x| \leq \frac{1}{2}, \, \, 0 \leq y \leq \frac{1}{2}\right\} ,$$
$$ \scalebox{0.95}{ $\mathcal{P}_3=\left\{x+iy \mid 0 \leq x ,  \frac{x}{\sqrt{3}}  \leq y \leq \frac{(1-x)}{\sqrt{3}} \right\} \cup \left\lbrace x+iy \mid 0 \leq x \leq \frac{1}{2} \, \, , \, \, \frac{-x}{\sqrt{3}} \leq y \leq \frac{x}{\sqrt{3}} \right\rbrace $}.$$
For any $d \notin \{1,3\}$, we have
$$\begin{array}{lll}  \pd= \{x+iy \mid 0 \leq x \leq 1 \, , \, 0 \leq y \leq \sqrt{d} \}, & \mbox{if} & d \equiv 1,2 \, (\mbox{mod } 4),  \\[5pt] \pd=\{x+iy \mid 0 \leq x \leq 1 \, , \, 0 \leq y \leq \frac{\sqrt{d}}{2} \}, & \mbox{if} & d\equiv 3 \, (\mbox{mod } 4) . \end{array}$$

Note that  $ \SLOd_{\infty}= \{\gamma \in \SLOd \mid \gamma(\infty)=\infty \}$ acts by isometries in $\mathbb{C}$ and this action is cocompact. The sets $ \pd $ above  are fundamental domains for this induced action which contain the origin.

In analogy with the action of the modular group in the hyperbolic plane, we have the following theorem (see \cite{elstrodt}).

\begin{theorem} \label{fd} 		
	The set $\FD_d=\{(z,t) \in \B_d \mid z \in \pd\}$
	is a fundamental domain for the action of $\Gd$ on $\hip$.
\end{theorem}

\section{Height estimates for Bianchi groups}
\label{sec3}

We refer to \cite{pila} for the description of the \textit{absolute exponential height} $\Ht(x)$ of an algebraic number $x$. In the special case when $x \in \mathcal{O}_d$, the \textit{height} of $ x $ is given by $\Ht(x) = |x|$, the complex norm of $ x $. For an element $ T \in \SL(2,\mathcal{O}_d) $, the height $\Ht(T)$ of $T$ is the maximum of the heights of its entries. If $[M] \in \Gd$ with $M \in \SL(2,\mathcal{O}_d)$, then we can define $\Ht([M]):=\Ht(M),$ since $\Ht(T)=\Ht(-T)$  for any $ T \in \SL(2,\mathcal{O}_d) $. We note that for $ A,B \in \Gd $, we have $  \Ht(AB)\leq 2\Ht(A)\Ht(B).$

In \cite{pila2}, Pila defines the notion of intricacy of a point of a symmetric space with respect to a fundamental domain of a lattice acting on the space. If we consider for each $ d $ as above, the fundamental domain $ \F _d $ of the Bianchi group, we can extend this notion  by defining the $ \Gamma_d -  $\emph{intricacy} of a point $ (z,t) \in \hip $ with respect to $ \F _d $ by 
$$ \mathrm{I}_{\F _d}(z,t) = \min\{\mathrm{H}(\sigma)\}, $$
where $ \sigma \in \Gamma_d $ runs through the set of elements in $\Gamma_d$ which take $ (z,t) $ into $ \F_d $.

Motivated by the analogue problem solved by Habegger and Pila in \cite{pila}  we define the function $  D(z,t) = \max\{1, |z|, t^{-1}\}  $ on $ \hip $, and prove our main theorem.

\begin{theorem}
	\label{mainthm}
	Let $d>0$ be a squarefree integer. There exists a  constant $c(d)>0$ such that  $$ \mathrm{I}_{\F _d}(z,t) \leq c(d) D(z,t)^2$$ for any $(z,t) \in \hip$. Furthermore, the exponent $2$ is sharp  and 
 $c(d) \leq C d$ for some universal constant $C>0$.
\end{theorem}

Before proving Theorem \ref{mainthm} we will prove an auxiliary result.

\begin{lemma}
	\label{bezout}
	Let $ d>0 $ be a squarefree integer. Let $\alpha,\beta \in \Od\setminus \{0\}$. Suppose that $\langle \alpha, \beta \rangle = \Od$. There exist a constant $C_d \geq 1$ and $ x,y\in \mathcal{O}_d $ such that $$1 = \alpha x + \beta y \hspace{0.3cm} \mbox{with} \hspace{0.3cm} |x| \leq C_d | \beta| \, \, \mbox{and} \, \,  |y| \leq C_d |\alpha|.$$ 
\end{lemma} 
\begin{proof}
	For each $d$ we let $\f_d$ be the canonical fundamental domain of $\mathcal{O}_d$ as a lattice of $\mathbb{C}$, which is the parallelogram generated by $1$ and $\omega$.

 Let $\varepsilon_d = \mbox{diam}(\C/\Od)$ be the diameter of the torus $\C / \Od$. If  $(x_0,y_0) \in \mathcal{O}_d^2$ is a solution of the equation $1=\alpha x_0 + \beta y_0,$ then for any $\lambda \in \mathcal{O}_d$ the pair $(x_\lambda,y_\lambda)=(x_0-\lambda \beta, y_0+\lambda \alpha)$ is another solution.
	
	Consider now the complex number $\frac{x_0}{\beta}$. There must exist $z_0 \in \f_d$ and $\lambda_0 \in \mathcal{O}_d$ such that $\frac{x_0}{\beta}=z_0 + \lambda_0$ and  also  $\mu \in \Od$ with $|z_0-\mu| \leq \varepsilon_d$. Therefore,  $\lambda=\lambda_0+\mu \in \mathcal{O}_d$ satisfies $\left|\frac{x_0}{\beta}-\lambda\right| \leq \varepsilon_d$, which is equivalent to $|x_0-\lambda \beta| \leq \varepsilon_d |\beta|.$
	
	For this $\lambda$, the pair $(x_\lambda,y_\lambda)$ satisfies  $1=\alpha x_{\lambda} + \beta y_{\lambda}$ and the following inequalities hold: 
	$$|x_\lambda| \leq \varepsilon_d |\beta| \, \, \mbox{and} \,\, |y_\lambda|=\frac{|1-\alpha x_\lambda|}{|\beta|} \leq 1+|\alpha|\varepsilon_d \leq (1+\varepsilon_d) |\alpha|.$$ 
	The lemma is then proven if we take $C_d=1+\varepsilon_d.$
\end{proof}

\begin{rmk}
Note that $\varepsilon_d \leq \mathrm{diam}(\f_d)$ for all $d$. Hence, there exists a universal constant $c_1 >0$ such that $C_d \leq c_1 \sqrt{d}$ for all $d$.
\end{rmk}

We can now prove our main theorem.

\begin{proof}[Proof of Theorem \ref{mainthm}]
	We fix $d>0$. Our goal is to give information about some element of $\Gd$ which sends an arbitrary point $(z,t) \in \hip$ into the fundamental domain $\FD_d$. By \cite[Lemma 3.3]{elstrodt}, this can be done in two steps. In the first step, we need to construct  $\tau \in \SLOd$ such that  $\tau (z,t)=(z',t') \in \B_d$. After it, we use the invariance of the action of $\SLOd_\infty$ in $\B_d$ in order to construct $\sigma \in \SLOd_\infty$ such that $\sigma \tau (z,t)=\sigma(z',t')=(z'',t')$ with $(z'',t') \in \FD_d$. 

	We first give a bound for the height of $\tau$ in terms of $(z,t)$. Define 
	$$\mu'_d(z,t) = \displaystyle \min_{\langle \gamma, \delta \rangle = \Od} \frac{|\gamma z + \delta|^2 + |\gamma|^2t^2}{t}.$$
	By \cite[Lemma 3.3 - Chapter $7$]{elstrodt}, an element  $\tau = \left[ \begin{array}{cc} \alpha_0 & \beta_0 \\ \gamma_0 & \delta_0 \end{array}\right] $  of  $\SLOd$ satisfies $\tau (z,t)=(z',t') \in \B_d$ if and only if
	\begin{equation}\label{eqsol}
	\mu_d'(z,t)t=|\gamma_0 z + \delta_0|^2 + |\gamma_0|^2t^2.
	\end{equation}
    Consider a pair $(\gamma_0,\delta_0) \in \Od^2$ with $\langle \gamma_0, \delta_0 \rangle = \Od$ satisfying Equation \eqref{eqsol}.
	
	If $\gamma_0=0$ or $\delta_0=0$, then either $\delta_0$ or $\gamma_0$ must be a unit of $\Od$. It is then easy to find some $\tau =\left[ \begin{array}{cc}  \alpha_0 & \beta_0 \\ \gamma_0 & \delta_0                                                                                                                   \end{array} \right] \in \SLOd$ with $\Ht(\tau)=1$ and $\tau (z,t) \in \B_d$. In this case we have the trivial inequality $\Ht(\tau) \leq D(z,t)$.
	
	Suppose now that $\gamma_0, \delta_0 \in \Od \setminus \{0\}$. By Lemma \ref{bezout}, there exist a constant $C_d$ and integers $\alpha_0,\beta_0 \in \Od$ such that $\tau = \left[ \begin{array}{cc}  \alpha_0 & \beta_0 \\ \gamma_0 & \delta_0                                                                                                                  
	\end{array} \right] \in \SLOd$
	satisfies $|\alpha_0| \leq C_d|\gamma_0|$ and $|\beta_0| \leq C_d|\delta_0|$, i.e.  $$\Ht(\tau) \leq C_d\max\{|\gamma_0|,|\delta_0|\}.$$
	
	By construction, $\tau (z,t)=(z',t') \in \B_d$ with $$t'=\frac{t}{|\gamma_0z+\delta_0|^2 + |\gamma_0|^2t^2}=\frac{1}{\mu'_d(z,t)} \geq t.$$ Hence $t\mu'_d(z,t) \leq 1$ and then $$\mu'_d(z,t) \leq t^{-1}.$$

	Moreover, from the definition of $ D(z,t) $ and from Equation \eqref{eqsol}, we obtain
	\begin{equation}
	\label{C13420.1}
		\Ht(\gamma_0) = |\gamma_0| \leq \sqrt{\frac{\mu'_d(z,t)}{t}}\leq t^{-1} 
		\leq D(z,t) 
	\end{equation}
		and
	\begin{equation}
			\Ht(\delta_0)  \leq |\gamma_0z|+|\gamma_0z+\delta_0| \leq |\gamma_0||z|+1 \leq  2\Ht(\gamma_0)D(z,t).
	\end{equation}
	
	Therefore,
	\begin{equation}
	\label{mtx}
	  \Ht(\gamma_0) \leq  \Ht(\tau) \leq 2C_d \Ht(\gamma_0)D(z,t).
	\end{equation}
	
	We now estimate the norm of $$ z'= \frac{(\alpha_0 z+\beta_0)(\bar{\gamma_0}\bar{z}+\bar{\delta_0}) + \alpha_0 \bar{\gamma_0}t^2}{\left|\gamma_0z + \delta_0 \right|^2 +\left|\gamma_0\right|^2 t^2}.$$
	
	We note first that
\begin{align*}
 |\alpha_0z+\beta_0||\gamma_0 z+\delta_0| & \leq \left|\alpha_0 \left( z+\frac{\delta_0}{\gamma_0} \right) - \frac{1}{\gamma_0}\right||\gamma_0z+\delta_0|  \\
  & \leq \frac{|\alpha_0|}{|\gamma_0|}|\gamma_0z+\delta_0|^2 + \frac{|\gamma_0z+\delta_0|}{|\gamma_0|} & \\ 
  & \leq \frac{|\alpha_0|}{|\gamma_0|}|\gamma_0z+\delta_0|^2 + |\gamma_0z+\delta_0| .
\end{align*}

We have $ |\alpha_0| \leq C_d|\gamma_0|$ and therefore,
	\begin{align*}
	|z'| &  \leq \frac{|\alpha_0||\gamma_0z+\delta_0|^2}{|\gamma_0||\gamma_0z+\delta_0|^2}+ \frac{|\gamma_0z+\delta_0|}{|\gamma_0z+\delta_0|^2+ |\gamma_0|^2t^2} + \frac{|\alpha_0|}{|\gamma_0|} \\
	      & \leq 2C_d + \frac{1}{2\Ht(\gamma_0)t}.  
	\end{align*}
We observe that we have used in the last inequality the fact that for every real numbers $a,b$ with $b>0$ it holds that $\frac{a}{a^2+b^2} \leq \frac{1}{2b}$, as $a^2+b^2 \geq 2ab$. 

By \eqref{C13420.1} and by noticing that $C_d \geq 1,$ we obtain $2C_d \le \dfrac{2C_d}{\Ht(\gamma_0)t}$.
Thus,
\begin{align}\label{normzprime}
|z'| \leq  3C_d\frac{1}{\Ht(\gamma_0)t} \leq   3C_d\frac{D(z,t)}{\Ht(\gamma_0)}.
\end{align}

	If $z'$ is in $\pd$, then the proof of first part finishes here, by using \eqref{C13420.1} and \eqref{mtx}. Otherwise, if we repeat the argument of the proof of Lemma \ref{bezout} we can show that there exists 
	$\sigma_{z'} \in \SLOd_\infty$ with $\sigma_{z'}(z') \in \pd$ and 
	\begin{equation} 
	\label{ineqb}
	\Ht(\sigma_{z'}) \leq |z'|+ C_d.
	\end{equation}

	Therefore, the following inequality follows from \eqref{C13420.1}, \eqref{mtx}, \eqref{normzprime}, \eqref{ineqb}, and from the fact that $\Ht(\sigma_{z'}\tau)\leq 2\Ht(\sigma_{z'})\Ht(\tau)$:

\[\mathrm{I}_{\F _d}(z,t) \leq 16C_d^2 D(z,t)^{2}.\]
	
If we set $c(d)=16C_d^2$, the first part of the theorem is proved for any $d$. Moreover, we already know that $C_d \leq c_1\sqrt{d}$  for a universal constant $c_1>0$. Hence, $c(d) \leq Cd$, where $C>0$ is a universal constant.

In order to finish the proof, we show that exponent $2$ is sharp. Indeed, for every integer $ n\geq 2 $, we take the sequence of points $(z_n, t_n) := (\frac{2n^2-1}{2n}, \frac{1}{2n}) \in \hip $. It is straightforward to check that $ \sigma_n := \left[ \begin{array}{cc}
	n & 1-n^2\\
	-1 & n
	\end{array} \right] \in  \Gamma_d$  and satisfy
	$$  \sigma_n(z_n, t_n) = (0,n) \in \FD_d, $$
	for every $ d $. 
	We have $ D(z_n, t_n) = 2n $ and  $  \Ht(\sigma_n) = n^2-1 $.
	 
	Note that $\mathrm{I}_{\F _d}(z_n,t_n) $ is not necessarily equal to $H(\sigma_n)$. However, for each $n$, by continuity of $\log(D(z,t))$, there is an open set $U_n\subset \hip$ containing $(z_n,t_n)$ so that for any $(w,s)\!\in  U_n$, it holds that $|\log(D(w,s))-\log(D(z_n,t_n)|  <\log2.$
	Since $\sigma_n^{-1}$ is continuous in $(0,n)$ and $(z_n,t_n)=\sigma_n^{-1}(0,n)$, there exists an open neighborhood $V_n$ of $(0,n)$ such that if $(z,t) \in V_n$, then $\sigma_n^{-1}(z,t) \in U_n$. Therefore if we take $(\epsilon_n,n) \in V_n \cap \mathrm{int}(\F_d)$ and define $(z'_n,t'_n)=\sigma_n^{-1}(\epsilon_n,n)$  we conclude that $(z'_n,t'_n) \in U_n$ and thus  satisfies  $ 2 \ge \frac{D(z'_n,t'_n)}{D(z_n, t_n)} \ge \frac{1}{2}$ , with  $\mathrm{I}_{\F _d}(z'_n,t'_n)  = n^2-1$. Thus we get
	\[ \mathrm{I}_{\F _d}(z'_n,t'_n)  \geq \frac{1}{8} D(z_n, t_n)^2 \ge \frac{1}{32} D(z'_n,t'_n)^2 \mbox{ for all } n  \mbox{ and } D(z'_n, t'_n) \to \infty.\]
\end{proof}

\begin{rmk}
	If $ h(d) = 1 $, then there exists an optimal constant $ t_d>0 $ such that the set  $ \Sigma_d  = \pd \times [t_d, + \infty) $ contains the fundamental domain $ \FD_d $ and  \begin{equation*}
	     \{ \gamma \in \Gamma_d \mid \gamma \FD_d \cap \Sigma_d \neq \emptyset \} \subset \{\gamma \in \Gamma_d \mid \Ht(\gamma) \leq c(d)s(d)^2 \}, 
	\end{equation*}   
	where $ s(d) $ is the maximum value of the function $ D(z,t) $ in $ \Sigma_d $.
	
	We observe that the latter set is finite and this proves that $ \Sigma_d $ is a fundamental set for $ \Gamma_d $. We show in Section \ref{counting} how to estimate the cardinality of this finite set. 
\end{rmk}

\section{On the reduction theory of binary Hermitian forms}
\label{reductionhforms}

We start this section with a brief description of binary Hermitian forms. We are interested in the Reduction of binary Hermitian forms over $ \Od $. For a more complete description of this subject, the reader can refer to \cite{beshaj}.

For any subring $ R $ of $ \C $, we can consider the set of Hermitian matrices $$  H(R) = \{A \in M_2(R) \mid  A^* = A \} . $$

Each $ f = \left( \begin{array}{cc}
a & b \\ 
\bar{b} & d\\
\end{array} \right) \in H(R) $ defines a Hermitian form in  $ \C $ by the following
$$ f(X,Z) = aX\bar{X} +2\Re(bX\bar{Z}) +dZ\bar{Z}, $$
and the discriminant of $ f $ is defined by $ \Delta(f) = ad-|b|^2 $.

If $ f(X,Z) >0 $ for any $ (X,Z) \in \C \times\C \backslash \{(0,0)\} $, then $ f $ is said to be \textit{positive definite}. 
We observe that if $ a\neq 0 $, then $ f(X,Z) = a\Big( |X +\frac{bZ}{a}|^2 + \frac{\Delta}{a^2}|Z|^2\Big) $. Hence we conclude that $f$ is positive definite if and only if $ a>0 $ and $ \Delta>0 $.

Let $ H^+(R) $ denote the set of positive definite Hermitian forms and consider its quotient by the action of $ \R^{>0} $ by scalar multiplication, which we denote by $ \tilde{H}^+(R) = H^+(R)/\R^{>0} $.
The following map gives a bijection between $ \tilde{H}^+(\C) $ and the hyperbolic $ 3- $space $ \hip $:
\begin{eqnarray*}
	\xi:\tilde{H}^+(\C)& \to & \hip \\
	f  & \mapsto& \Big(-\frac{b}{a}, \frac{\sqrt{\Delta(f)}}{a}\Big).
\end{eqnarray*}

This map is a bijection, with inverse map given by 
$$ \xi^{-1}(z,t) =  \left\lbrace \lambda
\left(
\begin{array}{cc} 1 & -z\\
-\bar{z} & \, \, \,|z|^2+t^2
\end{array}\right)
\mid \, \lambda \in \R^{>0} \right\rbrace $$
and it is also $ \mathrm{SL}(2,\C)-$equivariant.  In other words, if we define the action of $ \mathrm{SL}(2,\C) $ on $ H^+(\C) $ by $ \rho(g)f (X,Z) = f(g^{-1}(X,Z))$,  then $ \xi(\rho(g)f) = g\xi(f)$, for every $ g \in \mathrm{SL}(2,\C) $ and every $ f \in  H^+(\C) $.

For $ R= \Od $, we say that $ f \in H^+(\Od) $ is  a reduced form if $ \xi(f) \in \mathcal{F}_d $, the fundamental domain for the Bianchi group $ \Gamma_d $.

The height of a positive definite binary Hermitian form is given by 
$$ \Ht(f) = \max\{a, |b|, d\}. $$
For any $\Delta >0 $ and any subring $R \subset \C$, we define $H^+(R, \Delta)=\{ f \in H^+(R) \mid \Delta(f)=\Delta \}.$

\begin{lemma} Let $ D = D(z,t) $ be the function defined in Section \ref{sec3}. 
	If $ f  \in H^+(R, \Delta) $, then
	$$ D(\xi(f)) \leq \frac{\Ht(f)}{\sqrt{\Delta}}. $$
\end{lemma}

\begin{proof}
	Consider $f= \left( \begin{array}{cc}    a & b \\ \bar{b} & d \end{array}\right) \in H^+(R, \Delta)$ and $D(\xi(f))=\max\{1,\frac{|b|}{a}, \frac{a}{\sqrt{\Delta}} \}$. 	We look at the following cases: 
	\begin{itemize}
		\item $ \Ht(f) = |b| $;\\
		This case cannot happen, because then we would have $ a\leq |b| $ and $ d\leq |b| $, what would imply $ ad\leq |b|^2 $ and then $ \Delta \leq 0 $, which is a contradiction.
		
		\item $ \Ht(f) = a $;\\
		Then we would get $ D(\xi(f)) = \max\{1, \frac{a}{\sqrt{\Delta}} \} $, as $ \frac{|b|}{a} \leq 1 $. From $ d\leq a $, we obtain 
		$$ ad\leq a^2 \leq a^2+|b|^2 \Rightarrow \Delta \leq a^2, $$
		and thus $$  D(\xi(f)) = \frac{a}{\sqrt{\Delta}} = \frac{\Ht(f)}{\sqrt{\Delta}}. $$
		
		\item $\Ht(f)=d$;\\
		In this case we look separately at the following options: 
		
		If $ |b| \leq a \leq d $, we get that $ D(\xi(f)) =\max\{1, \frac{a}{\sqrt{\Delta}}\}$.
		We have then $ \frac{a}{\sqrt{\Delta}} \leq  \frac{\Ht(f)}{\sqrt{\Delta}} $. On the other hand, $ d^2 \geq ad \geq ad-|b|^2 = \Delta $. Hence $ 1\leq  \frac{d}{\sqrt{\Delta}} = \frac{\Ht(f)}{\sqrt{\Delta}} $ and we conclude that $ D(\xi(f)) \leq \frac{\Ht(f)}{\sqrt{\Delta}} $ as stated.
		
		It remains to consider $ a \leq |b| \leq d $, which is equivalent to showing the inequality $$\max \left\{ \sqrt{\Delta}, \frac{|b|\sqrt{\Delta}}{a},a \right\} \leq d. $$
		
		By our assumption, we have $$\sqrt{\Delta} \leq \sqrt{ad} \leq d \hspace{0.4cm} \mbox{and} \hspace{0.4cm} a \leq d.$$
		
		We finish the proof if we show that $|b|\sqrt{\Delta} < ad$. Note that the real quadratic function $f(x)=x^2-|b|^2x+|b|^4$ is strictly positive. In particular, $$f(ad)=(ad)^2-|b|^2(ad)+|b|^4 > 0 \Leftrightarrow |b|\sqrt{\Delta} < ad.$$

	\end{itemize}

\end{proof}

As a consequence of our main result and lemma above, we get the following
\begin{corollary} Let $ d $ be a squarefree integer. Given a positive integer $ \Delta > 0$ and $ f \in H^+(\Od, \Delta), $ there exists  $ g \in  \mathrm{SL}(2,\Od) $ such that $ \rho(g)f$ is a reduced form and which satisfies
	$$  \Ht(g) \leq c(d)D(\xi(f))^{2} \leq \frac{c(d)\Ht(f)^{2}}{\Delta},$$
	where $ c(d) $ is the constant in Theorem \ref{mainthm}.
\end{corollary}

\section{Counting elements with bounded height in Bianchi groups.}
\label{counting}

In what follows, the relation $f(t) \precsim g(t)$ for two positive functions $f,g$ means that  $\displaystyle \limsup_{t \to \infty} \frac{f(t)}{g(t)} \leq 1$. We write $f(t) \sim g(t)$ if  $f(t) \precsim g(t)$ and $g(t) \precsim f(t)$.

Given a linear algebraic group $G$ defined over $\mathbb{Z}$, a natural problem is to investigate the asymptotics as $T \to \infty$ of $$\# \{A \in G(R) ~|~ \Ht(A) \leq T\},$$ 
where $R$ is the ring of integers of a number field. In \cite{sarnak} this problem is studied in a more general context but they consider only the ring $R=\mathbb{Z}$. As an application of their main result, the authors consider  $G=\mathrm{SL}_n$ [op. cit., Example 1.6] and give an explicit constant $c(n)>0$ such that $$ \# \{ A \in \mathrm{SL}(n,\mathbb{Z}) \mid H(A) \leq T\} \sim c(n)T^{n^2-n}. $$ 

We are interested in this problem for $G=\mathrm{SL}_2$ and $R$ being the ring of integers of an imaginary quadratic extension of $\mathbb{Q}$. In fact, we will count elements in $\mathrm{PSL}_2$, this is not a matter since the number of matrices of given height in $\mathrm{SL}(2,R)$ is twice the number of elements in $\mathrm{PSL}(2,R)$ with the same height.

For any Bianchi group $\Gamma_d$ and positive number $T>0$ we consider the set $\W_d(T)=\{A \in \Gamma_d \mid \Ht(A) \leq T \}$.
We define the function $$\Nd(T)= \# \W_d(T).$$

It is known that $\Nd(T) < \infty$ for any $T>0$. Now we want to show that this function grows polynomially with degree $4$. 

Before proving the main result of this section, we recall a classical result of Schanuel \cite{sch}. Let $\mathcal{I}_d$ be the ideal class group  of $K_d$. The \emph{class} $[P]$ of a point $P=(x:y) \in \Pd$ is the class of the fractional ideal $\langle x, y \rangle$ in $\mathcal{I}_d$. Up to normalization, the definition of the height of $P$ is given by $$\Ht(P):=\frac{\max\{|x|,|y|\}}{\sqrt{\N(\langle x, y \rangle)}}.$$

Let $$\Xd(T)=\{ P \in \Pd \mid[P]=[(1:0)] \mbox{ and }  \Ht(P) \leq T\}.$$

It follows from \cite[Theorem 1]{sch} that the cardinality of $\Xd(T)$ has the following growth:

\begin{equation} \label{growx}
\# \Xd(T) \sim \tau(d)T^4,
\end{equation}
for some constant $\tau(d)>0$ which depends only on $d$. 
\begin{rmk}
    We observe that the height in \cite{sch} is normalized to be given by $\Ht(x) = |x|^2$ and this is why the exponent 4 appears instead of 2, as Theorem 1 therein suggests.
\end{rmk}

Consider the map $\phi_d: \Gamma_d \rightarrow \Pd$ given by $$\phi_d \left[ \begin{array}{cc}
\alpha & \beta \\ \gamma & \delta                                                                                              \end{array} \right] = (\alpha : \gamma).$$

Note that $[\phi_d(\tau)]=[(1:0)]$ for any $\tau \in \Gamma_d$ since $\langle \alpha, \gamma \rangle = \Od$. 
Moreover, the map $\phi_d$ satisfies
\begin{equation} \label{C10420.1}
 \Ht(\phi_d(\tau)) \leq \Ht(\tau) \mbox{ for all } \tau \in \Gd.
\end{equation}

Consider the group $S=\mathbb{Z}/2\mathbb{Z} \times \mathbb{Z}/2\mathbb{Z}$ and let $\iota=(1,0),\kappa=(0,1)$ be the canonical generators of $S$. We define an action of $S$ in $\Gd$ by $\iota \cdot \tau = \tau^{-1} $ and $\kappa \cdot \tau=\tau'$ for any $\tau \in \Gd$, where $\tau'$ denotes the class of any transpose of a matrix which represents $\tau$. Since inversion and transposition commutes and since they are involutions, we have a well defined action of $S$ on $\Gd$. Note that 
 \[\Ht(s \cdot \tau) = \Ht(\tau) \mbox{ for any } s \in S \mbox{ and } \tau \in \Gd.  \] 

Note that for any $\tau \in \Gd$, there exists $s \in S$ such that the height of $s \cdot \tau$ is realised at the first column of any representative of $\tau$, i.e. $\Ht(s \cdot \tau)=\Ht(\phi_d(s \cdot \tau))$. Thus, for each $S$-orbit in $\Gd$ we can choose some element where the inequality \eqref{C10420.1} is an equality.

For each $T>0$ define $\tilde{W}_d(T)=\{ \tau \in W_d(T) \mid \Ht(\tau)=\Ht(\phi_d(\tau))\},$ and denote its cardinality by $\tilde{N}_d(T)$. Note that $W_d(T)$ is $S$-invariant for any $T>0.$ By paragraph above, if we decompose $W_d(T)$ in $S$ orbits, then for any orbit we can choose an element in $\tilde{W}_d(T)$. Since each orbit has at most $4$ elements we conclude that 
\begin{equation} \label{C10420.2}
 N_d(T) \leq 4 \tilde{N}_d(T) \mbox{ for any } T>0.
 \end{equation}

By fixing this  notation, we can now prove the main theorem in this section.

\begin{theorem}
	Let $d>0$ be a squarefree integer. There exist constants $0< l_1(d) \leq l_2(d)$ such that $$l_1(d)T^4 \precsim
	\Nd(T) \precsim l_2(d)T^4,$$ 
	for all sufficiently large $T$.
\end{theorem}
\begin{proof}
	We fix an arbitrary squarefree integer $d>0$.  	Let $\Ud=\{ P \in \Pd \, |\, \, [P]=[(1:0)] \}$. We  have already observed that the image of  $\phi_d$ is in $\Ud$. Now we define a right inverse for $\phi_d$. For each $P \in \Ud$ there exist $\alpha, \gamma \in \Od$ such that $\langle \alpha, \gamma \rangle = \Od$ and $P=(\alpha:\gamma)$. Indeed, if $P=\left(\frac{p}{q}:\frac{r}{s}\right),$ then $$\left\langle \frac{p}{q},\frac{r}{s} \right\rangle = \frac{u}{v} \Od,$$ 
	where $p,r \in \Od$ and $q,s,u,v \in \Od \setminus \{0\}$.  If we define $\alpha'=pvs, \gamma'=rqv, \xi=uqs $ we get $P=(\alpha':\gamma')$ and $\langle \alpha', \gamma' \rangle = \xi \Od.$ Finally, by writting $\alpha=\frac{\alpha'}{\xi}$ and $\gamma=\frac{\gamma'}{\xi}$, we get that $\alpha,\gamma \in \Od$, $P=(\alpha : \gamma)$ and $\langle \alpha, \gamma \rangle = \Od.$
	
	Consider $P=(\alpha:\gamma) \in \Ud$. If $\gamma=0$  we define $\psi_d(P)=\mathrm{I} \in \Gamma_d$, if $\alpha=0$ we define $\psi_d(P)=\left[ \begin{array}{cc} 0 & -1 \\ 1 & 0 \end{array} \right] \in \Gamma_d$. If $\alpha, \gamma \in \Od \setminus \{0\}$ we can use Lemma \ref{bezout} in order to obtain a constant $C_d \geq 1$ and elements $\beta, \delta \in \Od$ with $|\beta| \leq C_d|\alpha|, |\delta| \leq C_d|\gamma|$ and such that $\psi_d(P)=\left[ \begin{array}{cc} \alpha & \beta \\ \gamma & \delta \end{array} \right] \in \SLOd$.
	
	Therefore,
	$$\begin{array}{lllll} \phi_d (\psi_d (P)) & = & P & \mbox{for all} & P \in \Ud.  \\[5pt] \Ht(\psi_d(P)) & \leq & C_d \Ht(P) & \mbox{for all} & P \in \Ud. \end{array}$$
	
	Hence, for any $T>0$ the restriction $\psi_d^T$ of $\psi_d$ to $\Xd \left(\frac{T}{C_d}\right)$ is an injective map whose image is contained in  $\W_d(T)$. Thus, $$\#\Xd \left(\frac{T}{C_d}\right) \leq \Nd(T), \, \, \mbox{for all} \, \, T>0.$$ 
	
	By \eqref{growx}, if we define $l_1(d)=\frac{\tau(d)}{C_d^4}$, then $l_1(d)T^4 \precsim \Nd(T).$
	
In order to give an upper bound for $N_d(T)$ we will first estimate $\tilde{N}_d(T)$ from above. 
For each $m \in \mathbb{N},$ let $\phi_{d,m}$ be the restriction of $\phi_d$ to the set $\mathrm{Dom}(\phi_{d,m})=\{ \tau \in \Gd \mid  \Ht(\tau)=\Ht(\phi_d(\tau))=\sqrt{m}\}$. Note that 
\begin{equation}
\label{ntil}
    \tilde{N}_d(T)=\sum_{m \le T^2} \# \mathrm{Dom}(\phi_{d,m}).
\end{equation}

We now show that for any $m,d$ and $\tau \in \mathrm{Dom}(\phi_{d,m})$, the  cardinality of the set $\{\tau' \in \mathrm{Dom}(\phi_{d,m}) \mid \phi_{d,m}(\tau')=\phi_{d,m}(\tau)\}$ is bounded by a universal constant. 

If $\phi_{d,m}(\tau)=\phi_{d,m}(\tau')$, then  $\tau^{-1}\tau' \in \SLOd_\infty.$ 

Consider matrices $A=\begin{pmatrix} \alpha & \beta \\
\gamma & \delta \end{pmatrix}, A'=\begin{pmatrix} \alpha' & \beta' \\ \gamma' & \delta' \end{pmatrix} \in \mathrm{SL}(2,\Od)$ such that $\tau=[A], \tau'=[A']$. Then there exist $\lambda \in \Od^*$  and $\mu \in \Od$ such that $A'=A \begin{pmatrix} \lambda & \mu \\
0 & \lambda^{-1} \end{pmatrix}$ and we can write
\begin{align}\label{C11420.1}
    \beta'=\alpha\mu+\beta\lambda^{-1} \mbox{ and } \delta'=\gamma\mu+\delta\lambda^{-1}.
\end{align}
Since $\Ht(A)=\Ht(\tau)=\Ht(\phi_d(\tau))=\max\{|\alpha|,|\gamma|\}$, we have two possibilities, either $|\alpha|=\sqrt{m}$ or $|\gamma|=\sqrt{m}$. If $|\alpha|=\sqrt{m},$ then we get $|\mu|\leq 2$. Indeed,
\begin{equation*}
    \sqrt{m} |\mu| =  |\alpha\mu| \leq |\beta'|+|\beta\lambda^{-1}| \leq  2\sqrt{m},
\end{equation*}
since  $|\beta|,|\beta'|\leq \sqrt{m} \mbox{ and } |\lambda^{-1}|=1.$
The same holds if $|\gamma|=\sqrt{m}$.

Hence there are at most $\frac{1}{2}(\#\Od^*)( \#\{u \in \Od \mid |u| \leq 2 \})$ possibilities for $\tau^{-1}\tau'$. Note that for any  $d$, we have $\#\Od^* \leq 6$ and $ \#\{u \in \Od \mid |u| \leq 2 \} \leq c_1$ , where $c_1>0$ is a universal constant.

By the definition of the domain of $\phi_{d,m}$, its image is contained in the set $\{P \in \Ud \mid \Ht(P)=\sqrt{m}\}$. Then for any $m \geq 1$ we have
\[\# \mathrm{Dom}(\phi_{d,m}) \leq 3c_1 \# \{P \in \Ud \mid \Ht(P)=\sqrt{m}\}.\]

Hence for any $T \ge 1$ we get by \eqref{ntil} 
\[\tilde{N}_d(T)  \leq 3c_1 \sum_{m \leq T^2}  \# \{P \in \Ud \mid \Ht(P)=\sqrt{m}\} = 3c_1\# X_d(T).         \]

Therefore, by \eqref{growx}, \eqref{C10420.2} and the inequality above we obtain \[\Nd(T) \precsim 12c_1\tau(d)T^4.\]
We finish the proof by observing that the constant $l_2(d)=12c_1\tau(d)$ depends only on $d$.

\end{proof}

\medskip

\noindent\textbf{Acknowledgments.}
We would like to thank Mikhail Belolipetsky for his encouragement and interest in this work. We also thank the referee for numerous comments and suggestions, including pointing out gaps and  possible improvements on our main result. We are grateful to  \textit{Universidade de S\~ao Paulo} and \textit{Universidade Federal do  Esp\'irito Santo} for the support and hospitality.
Besides, the second author acknowledges the University of Lille for receiving her in her postdoctoral stage and for the good interchange of ideas that she had there.


\begin{thebibliography}{}



\bibitem{beshaj} L. Beshaj,  Reduction theory of binary forms, {\em NATO Sci. Peace Secur. Ser. D Inf. Commun. Secur.} {\bf 41}, 84--116 (2015).
\bibitem{bianchi} L. Bianchi, Sui gruppi di sostituzioni lineari con coefficienti appartenenti a corpi quadratici immaginarî, {\em Mathematische Annalen} {\bf 40}, 332--412  (1892).
\bibitem{Borel} A. Borel, Commensurability classes and volumes of hyperbolic 3-manifolds, {\em Ann.
	Scuola Norm. Sup. Pisa Cl. Sci.} {\bf 8}, 1--33  (1981).
\bibitem{sarnak} W. Duke, Z. Rudinick, and P. Sarnak, Density of Integer Points on Affine Homogeneous Varieties, {\em Duke Math. J.} {\bf 71}, 143--179  (1993).
\bibitem{elstrodt} J. Elstrodt, F. Grunewald,  and J. Mennicke, {\em Groups acting on Hyperbolic Space - Harmonic Analysis and Number Theory}. Springer Verlag, 1998.
\bibitem{EGM89} J. Elstrodt, F. Grunewald,  and J. Mennicke, Some remarks on discrete subgroups of $\mathrm{PSL}(2,\mathbb{C})$, {\em J. Soviet. Math.} {\bf 46}, 1760--1788  (1989).
\bibitem{pila} P. Habegger and J. Pila, Some unlikely intersections beyond Andr\'e-Oort, {\em Compos. Math.} {\bf 148}, 1--27 (2012). 
\bibitem{MR03} C. Maclachlan and A.~W. Reid, {\em The arithmetic of hyperbolic 3-manifolds}, Graduate Texts in Mathematics, {\bf 219}, Springer, New York, (2003).
\bibitem{M15} D.~W. Morris, {\em Introduction to arithmetic groups}, Deductive Press, (2015).
\bibitem{martinorr} M. Orr, Height bounds and the Siegel property, {\em Algebra Number Theory} {\bf 2}, 455--478 (2018).
\bibitem{gisele} G.~T. Paula, Comparison of Volumes of Siegel Sets and Fundamental Domains  for $\mathrm{SL}_n (\mathbb{Z})$, {\em Geometriae Dedicata} {\bf 199}, 291--306 (2019).
\bibitem{sch} S. Schanuel, On heights in number fields, {\em Bull. Amer. Math. Soc.} {\bf 70}, 262--263 (1964).
\bibitem{pila2} J. Pila, O-minimality and the Andr\'e-Oort conjecture for $ \C ^n $, {\em Ann. of Math. (2)}  {\bf 173}, 1779--1840  (2011).
\end{thebibliography}
\end{document}